\documentclass{article}
\usepackage{amsmath}
\usepackage{amsfonts}
\usepackage{amsthm}
\usepackage{amssymb}
\usepackage{color}
\usepackage{enumerate}
\usepackage{tikz}
\usetikzlibrary{patterns}
\usetikzlibrary{decorations.pathreplacing,angles,quotes}
\usetikzlibrary{arrows}
\usetikzlibrary{calc,automata,patterns,decorations,decorations.pathmorphing}
\usetikzlibrary{fadings}
\usepackage{pgfplots}
\pgfplotsset{compat=newest}
\usepgfplotslibrary{fillbetween}

\definecolor{blue-green}{rgb}{0.0, 0.87, 0.87}
\definecolor{OwlYellow}{RGB}{ 242, 147,  24}
\definecolor{OwlRed}{RGB}{255,92,168}
\definecolor{OwlGreen}{RGB}{90,168,0}
\definecolor{OwlBlue}{RGB}{0,152,233}
\definecolor{arylideyellow}{rgb}{0.91, 0.84, 0.42}
\definecolor{aureolin}{rgb}{0.99, 0.93, 0.0}
\definecolor{wenge}{rgb}{0.39, 0.33, 0.32}
\definecolor{sealbrown}{rgb}{0.2, 0.08, 0.08}
\definecolor{salmonpink}{rgb}{1.0, 0.57, 0.64}
\definecolor{americanrose}{rgb}{1.0, 0.01, 0.24}
\definecolor{aquamarine}{rgb}{0.5, 1.0, 0.83}
\definecolor{babyblue}{rgb}{0.54, 0.81, 0.94}

\usepackage{ulem}  
\usepackage{graphicx}
\usepackage[colorlinks=true]{hyperref}
\hypersetup{urlcolor=blue, citecolor=sealbrown!30!green, linkcolor=blue}

\newtheorem{theorem}{Theorem}[section]
\newtheorem{lemma}[theorem]{Lemma}
\newtheorem{proposition}[theorem]{Proposition}

\newtheorem{problem}{Problem}

\newtheorem{corollary}[theorem]{Corollary}

\theoremstyle{definition}
\newtheorem{definition}[theorem]{Definition}
\theoremstyle{remark}
\newtheorem{remark}[theorem]{Remark}

\theoremstyle{definition}
\newtheorem{example}[theorem]{Example}

\newcommand\remove[1]{}

\def\f2{\mathbb{F}_2}

\newcommand{\diam}{{\rm diam}\hskip0.02cm}

\newcommand{\co}{\mathrm{c}_0}

\begin{document}

\title{H\"older-contractive mappings, nonlinear extension problem and fixed point free results}

\author{Cleon S. Barroso}

\date{}
\maketitle

\noindent{\bf Abstract:} For a bounded closed convex set $K$, in this note, we study the FPP for $\alpha$-H\"older nonexpansive maps, i.e. mappings $T\colon K\to K$ for which $\|T x -Ty\| \leq\| x - y\|^\alpha$ for all $x, y\in K$, $\alpha\in (0,1)$. First, we note that only finite-dimensional spaces have the H\"older-FPP. Moreover, the unit ball $B_X$ of any infinite-dimensional space fails the FPP for H\"older maps with $\mathrm{d}(T, B_X)>0$, where $\mathrm{d}(T, K)$ denotes the minimal displacement of $T$. We further show that reflexivity and weak sequential continuity are sufficient conditions to capture fixed points of H\"older-Lipschitz maps with bounded orbits. Next we focus on the existence of fixed point free $\alpha$-H\"older  maps $T\colon K\to K$ with $\mathrm{d}(T, K)\leq \varphi(\alpha)$ where either $\varphi(\alpha)=0$ or $\varphi(\alpha)\to 0$ as $\alpha \to 1$. Interesting results are obtained for the spaces $\mathrm{c}$, $\co$, $\ell_1$ and $\ell_2$, and also for $L_p$-spaces with $p\in[ 1, \infty]$. We also study the problem in spaces containing copies of $\co$ and $\ell_1$. Some questions are left open. 

\medskip

\begin{large}


\section{Introduction}

Let $(X,\|\cdot\|)$ be a Banach space. A bounded closed convex subset $K$ of $X$ is said to have the {\it fixed point property} (FPP) when every nonexpansive (i.e., $1$-Lipschitz) mapping $T\colon K\to K$ has at least one fixed point. Also, $X$ is said to have the FPP when every such $K$ has it. The FPP has its genesis in 1965, when F. Browder \cite{Brd1, Brd2}, D. G\"ohde \cite{Goh} and W. Kirk \cite{Ki} proved the first positive results for Hilbert spaces, uniformly convex spaces and, more broadly, reflexive spaces with the normal structure. Since then its connections with various aspects of Banach space theory has attracted plenty of interest from a number of sources. The current state of the art is the result of many fundamental contributions, including Sine \cite{Si}, Soardi \cite{So}, Karlovitz \cite{Kar}, Maurey \cite{M}, Alspach \cite{A}, Dowling, Lennard and Turett \cite{DLT, DLT4}, Benavides \cite{D-B}, Lin \cite{Lin, Lin2}. We also refer to \cite{D-B,Ben-Japon, DLT0, DLT, DLT1, DLT2,DLT4, vD, H-L, HL-MJP,KS, Lin} as well as the references therein for several further important advances. 

The FPP has developed in several directions including two major branches. One is directly focused on properties of underlying spaces for which one can verify the FPP, while the other emphasizes studies on the FPP or weak-FPP for different classes of maps. The literature surrounding these branches is vast, and we refer the reader to \cite{GK} for a good monograph on them. We also refer for instance to the papers \cite{D-B, Lin, Lin2} for fundamental contributions related to the connections between reflexivity and FPP, and the works \cite{Ki, Ki2} concerning geometric aspects related to the FPP for the class of asymptotically nonexpansive mappings. 

\medskip 
In this note we consider the class of H\"older-Lipschitz type maps. 

\begin{definition}\label{def:1} Let $K$ be a nonempty subset of a metric space $(X, \rho)$. For $\alpha, L \in (0,\infty)$ and $\lambda \in (0,1)$, a mapping $T\colon K\to  X$ is called
\begin{itemize}
\item[\it{(1)}] $\alpha$-H\"older $\lambda$-contractive if $\rho(Tx, Ty)\leq \lambda \rho^\alpha(x, y)$ for all $x,y \in K$. 
\item[\it{(2)}] $\alpha$-H\"older nonexpansive if $\rho(Tx, Ty)\leq \rho^\alpha(x, y)$ for all $x, y\in K$. 
\item[\it{(3)}] $\alpha$-H\"older $L$-Lipschitz with $L>1$, if $\rho(Tx, Ty)\leq L \rho^\alpha(x, y)$ for all $x,y \in K$. 
\item[\it{(4)}] Asymptotically $\alpha$-H\"older nonexpansive if there is a decreasing null sequence $(\kappa_n)_n$ in $(0,1]$ so that $\rho(T^nx, T^ny)\leq (1 + \kappa_n) \rho^\alpha(x, y)$ for all  $x, y\in K,\, n\in \mathbb{N}$.
\item[\it{(5)}] Uniformly $\alpha$-H\"older $L$-Lipschitz (resp. nonexpansive, contractive) if
$L>1$ (resp. $L=1, L<1$) and $\rho(T^nx, T^ny)\leq L \rho^\alpha(x, y)$ for all $x, y\in K$,\, $n\in \mathbb{N}$.
\end{itemize}
\end{definition}

Clearly if $\alpha=1$ we have the maps that have received the most attention in the literature. The aim of this work is to explore the case $\alpha<1$, which is overall more interesting. In fact, when $\alpha>1$ and $K$ has finite diameter all $\alpha$-H\"older $L$-Lipschitz maps on $K$ are $L\diam(K)^{\alpha -1}$-Lipschitz in classical sense. In addition, when $K$ is a convex set in a Banach space $X$ this class of mappings consists only of constant maps (Proposition \ref{R:3sec2}). Moreover, it can be proven that if $K$ is a hyperconvex metric spaces with $\diam(K)\leq 2$ and $T\colon K\to K$ is a fixed-point free $\alpha$-H\"older nonexpansive mapping then necessarily $\alpha<1$ (cf. Proposition \ref{R:8sec3}). 

\medskip 

For a convex set $C\subset X$ let $\mathcal{B}(C)$ denote the family of all nonempty bounded closed convex subsets of $C$. Given any set $K\subset X$ we shall denote by $\digamma(T)$ and $\mathrm{d}(T,K)$, respectively, the fixed point set and the minimal displacement of a map $T\colon K\to K$. That is,
\[
\digamma(T)= \big\{ x\in K \,\colon \, x =Tx\big\} \quad \&\quad \mathrm{d}(T, K) = \inf_{x\in K}\| x - Tx\|.
\]

As far as we know, the first work to investigate the FPP for H\"older-type maps was \cite{Ki3}. In this work, W. Kirk called a mapping $T\colon K\to K$ as $h$-nonexpansive ($h>0$) if
\[
\| T(x) - T(y)\| \leq \max\{ h, \| x -y\|\}\quad\text{for all } x, y\in K.
\]
It is clear that $\alpha$-H\"older nonexpansive maps ($0<\alpha<1$) are $1$-nonexpansive in the above sense. In \cite[Theorem 2.3]{Ki3}, he proved under certain conditions that  $\mathrm{d}(T,K)\leq h$. As another interesting work, A. Wi\'snicki \cite{Wi} had shown under some geometric conditions (cf. \cite[Theorem 4.2]{Wi}) that if $K\in \mathcal{B}(X)$ is a   weakly compact set, then the fixed point set of asymptotically regular semigroups acting on $K$ are H\"older-Lipschitz retracts.

\medskip 

In analogy with the FPP, it is natural to consider the H\"older-FPP. Precisely, we say a set $K\in \mathcal{B}(X)$ has the $\alpha$-HFPP ($\alpha<1$) if any $\alpha$-H\"older nonexpansive mapping on $K$ has a fixed point. Similar definition for the whole space $X$ also applies. 

\smallskip 

As a further motivation for considering this property, we note that $\alpha$-HFPP implies the FPP (cf. Theorem \ref{R:6sec3}). Also, if $K$ is a convex subset of $B_X$ (the closed unit ball of $X$) and $T\colon K\to K$ is nonexpansive then (cf. Proposition \ref{R:1sec4}) for every $\alpha, \varepsilon\in (0,1)$ there is a $\alpha$-H\"older $(\varepsilon + \diam(K)^{1-\alpha})$-Lipschitzian map $T_\varepsilon \colon K\to K$ with  $\digamma(T_\varepsilon)=\digamma(T)$, $\mathrm{d}(T_\varepsilon, K) \leq \mathrm{d}(T, K)$ and $\| T(x) - T_\varepsilon(x)\| \leq \varepsilon$ for all $x\in K$. 

\medskip 

It turns out however that only finite-dimensional spaces satisfy the $\alpha$-HFPP. 

\smallskip 

\begin{proposition}\label{prop:1} Let $X$ be a Banach space and $\alpha\in (0,1)$. Then:
\begin{itemize}
\item[(i)] $X$ has the $\alpha$-HFPP if and only if it is finite dimensional. 
\item[(ii)] If $X$ is infinite-dimensional then there exists a $\alpha$-H\"older nonexpansive mapping $T\colon B_X\to B_X$ with $\digamma(T)=\emptyset$ and $\mathrm{d}(T, B_X)>0$. 
\end{itemize}
\end{proposition}

\smallskip 
 
One direction of (i) is obvious. The converse is straightforward and follows from Lin-Sternfeld's theorem \cite[Theorem 2]{LS}. Indeed, assume $X$ is infinite-dimensional and let $f\colon B_X \to B_X$ be any $L$-Lipschitz map with $\mathrm{d}(f, B_X)>0$. Let $B_X(r)$ be the closed ball in $X$ with center $0$ and radius $r>0$. For $0<r\leq (1/2L)^{\frac{1}{1-\alpha}}\wedge \mathrm{d}(f, B_X)$, define $F\colon B_X(r)\to B_X(r)$ by $F(x) = r f(x/r)$. Then $F$ is $\alpha$-H\"older nonexpansive and  $\mathrm{d}(F, B_X(r))>0$. As for (ii) notice that if $R\colon B_X\to B_X(r)$ is a $2$-Lipschitz retraction onto $B_X(r)$, $T(x)=F(R(x))$ is a $\alpha$-H\"older nonexpansive mapping on $B_X$ with $\digamma(T)=\emptyset$ and $\mathrm{d}(T, B_X)>0$. 

\medskip 

These facts naturally lead us to the main problems considered here:

\begin{problem} Let $X$ be an infinite dimensional Banach space and $\varphi\colon (0,1)\to \mathbb{R}_{+}$ be a real-valued 
function such that $\varphi(\alpha)\to 0$ as $\alpha\to 1$.
\begin{itemize}
\item[($\mathcal{Q}1$)] Under which conditions a (uniformly) $\alpha$-H\"older $L$-Lipschitz mapping $T$ on a closed convex subset $K$ of $X$ has a fixed point?
\item[($\mathcal{Q}2$)] Is there for any $\alpha\in (0,1)$ a set $K\in \mathcal{B}(X)$ and a $\alpha$-H\"older nonexpansive mapping $T\colon K\to K$ such that $\digamma(T)=\emptyset$ and $\mathrm{d}(T, K)\leq \varphi(\alpha)$?
\item[($\mathcal{Q}3$)] Is there for any $\alpha\in (0,1)$ a ({\it uniformly}) $\alpha$-H\"older nonexpansive mapping $T\colon B_X\to B_X$ such that $\digamma(T)=\emptyset$ and $\mathrm{d}(T, B_X)=0$?
\end{itemize}
\end{problem}
 
\medskip 

In the spirit of Schauder and Schaefer's fixed point results (cf. \cite[p. 502]{Evans}), ($\mathcal{Q}1$) seems not to be a hard problem to tackle. A priori though, it is not clear how to solve ($\mathcal{Q}2$) and ($\mathcal{Q}3$) in all their generality. Interestingly, although ($\mathcal{Q}3$) seems to be a bit challenging, we provide satisfactory answers in several situations. We stress that, to the best of our knowledge, the results displayed in this work are not available in the published literature. Nevertheless, they stem from the combination of known results.

\smallskip 

The paper is organized as follows. Section \ref{sec:2} deals with some preliminaries. Our main results are delivered in Sections \ref{sec:3}, \ref{sec:4} and \ref{sec:5}.

\newpage

{\bf Acknowledgements.}  A significant part of this work was written during the author's visit to the Mathematics Department of the Federal University of Amazonas (UFAM), from December 24, 2021 to January 11, 2022. For this reason, he wishes to thank Professors Fl\'avia Morgana and Jeremias Le\~ao for the kind invitation and support. Some results of this work were presented at the I Fortaleza Conference on Analysis and PDEs in the summer 2022. The author also thanks the organizers of the conference, Professors Gleydson Ricarte, P\^edra Andrade and Jo\~ao Vitor da Silva, for the invitation to give a talk on the subject of this paper. 

\medskip 

\section{Preliminaries}\label{sec:2}

We follow \cite{AK,FHHMZ, LT} for notation. For reader's convenience, we recall some basic facts from Banach space and metric space theory as well. Throughout this paper, $X$ will stand for a real Banach space, $\|\cdot\|$ the norm of $X$, $S_X$ and $B_X$ the unit sphere and the closed unit ball of $X$ respectively. Recall a sequence $(x_n)_{n=1}^\infty$ in $X$ is called a basic sequence (or simply {\it basic}) if it is a Schauder basis for its closed linear span $[x_n]$. A basic sequence $(x_n)_{n=1}^\infty$  is called $1$-{\it subspreading} if for every increasing sequence $(n_i)_{i=1}^\infty\subset \mathbb{N}$, $(x_{n_i})_i$ is $1$-dominated by $(x_n)_n$. That is, the linear mapping $\mathfrak{L}(x_i):= x_{n_i}$ has norm $\|\mathfrak{L}\|\leq 1$. Henceforth $\co$, $\mathrm{c}$, $\ell_\infty$ and $\ell_p$ denote the usual space of all real sequences $(t_j)_{j=1}^\infty$ such that $\lim_j t_j=0$, $(t_j)_{j=1}^\infty$ is convergent, $(t_j)_{j=1}^\infty$ is bounded, and $\sum_{j=1}^\infty |t_j|^p<\infty$.  In addition, $L_p$ denotes the usual Lebesgue space $L_p[0,1]$. As in \cite[p. 315]{GP} we say that $X$ is {\it separably Sobczyk} if every subspace of $X$ isomorphic to $\co$ is complemented. A metric space $(K,\rho)$ is called hyperconvex if  any family of closed balls $\{B(x_\alpha; r_\alpha)\colon \alpha\in \Lambda\}$ of $K$ satisfying $\rho(x_\alpha, x_\beta)\leq r_\alpha + r_\beta$ for all $\alpha, \beta\in \Lambda$, has nonempty intersection. It is known that $B_{\ell_\infty}$ as well as any order interval in $\ell_\infty$ are hyperconvex. Further, a Banach space $X$ is hyperconvex if and only if it is isometrically isomorphic to a $C(K)$ space with $K$ stonian. In particular, $\ell_\infty$ and $L_\infty$ are standard examples of hyperconvex Banach spaces. A subset $A$ of $K$ is called {\it admissible} if $A$ can be written as the intersection of a family of closed balls centered at points of $K$. Denote by $\mathcal{A}(K)$ the family of all admissible subsets of $K$, and by $\mathrm{cov}(A)$ the {\it cov-closure} of $A$, that is,
\[
\textrm{cov}(A)= \bigcap\Big\{ B\colon B\in \mathcal{A}(K),\, B\supseteq A\Big\}.
\]
Note that $\mathcal{A}(K)=\big\{ A\subset K\colon A=\textrm{cov}(A)\big\}$. Aronszajn and Panitchpakdi \cite{AP} proved the following important result. 

\begin{theorem}\label{thm:AP} A metric space $(K,\rho)$ is hyperconvex if and only if for every metric space $M$ in which $K$ embeds isometrically, $K$ is a nonexpansive retract of $M$. 
\end{theorem}

\medskip 

We shall also use the following versions James's non-distortion theorem \cite{DJLT, DLT1}.

\begin{lemma}\label{thm:J} Let $(X,\|\cdot\|)$ be a Banach space. The following are true:
\begin{itemize}
\item[(i)] $X$ contains a copy of $\ell_1$ if and only if, for every null sequence $(\delta_n)_{n=1}^\infty$ in $(0,1)$, there exists a sequence $(x_n)_{n=1}^\infty$ in $X$ such that, for all $(a_n)_{n=1}^\infty \in \ell_1$ and $k\in \mathbb{N}$,
\begin{equation}\label{eqn:A}
(1- \delta_k) \sum_{n=k}^\infty | a_n | \leq \Bigg\| \sum_{n=k}^\infty a_n x_n \Bigg\|\leq \sum_{n=k}^\infty | a_n|.
\end{equation}
\item[(ii)] $X$ contains a copy of $\co$ if and only if, for every null sequence $(\delta_n)_{n=1}^\infty $ in $(0,1)$, there exists a sequence $(x_n)_{n=1}^\infty $ in $X$ such that, for all $(a_n)_{n=1}^\infty \in \co$ and $k\in \mathbb{N}$,
\begin{equation}\label{eqn:B}
(1- \delta_k) \sup_{n\geq k}| a_n | \leq \Bigg\| \sum_{n=k}^\infty a_n x_n \Bigg\|\leq (1+ \delta_k)\sup_{n\geq k}| a_n|.
\end{equation}
\end{itemize}
\end{lemma}

\smallskip 


We close this section by highlighting some easily verified facts.

\begin{proposition}\label{R:3sec2} Let $K$ be a complete subset of a metric space $(X,\rho)$. Then a $\alpha$-H\"older $L$-Lipschitz mapping $T\colon K\to K$ with $\alpha>1$ has a fixed point if:
\begin{itemize}
\item[\it{(1)}] $L\in (0,1)$ and\, $ \rho(Tx, x)\leq 1$ for some $x\in K$.  
\item[\it{(2)}] $L\in (0,1)$ and\, $\limsup_{i\to \infty} \rho(T^ix, x)< 1$ for some $x\in K$.  
\item[\it{(3)}] $X$ is a Banach space, $K$ is bounded and $L< \diam(K)^{1-\alpha}$.
\end{itemize}
In addition, if $X$ is a normed space and $K\in \mathcal{B}(X)$ then $T$  is constant. 
\end{proposition}


\smallskip 


\section{H\"older FPP in reflexive and hyperconvex spaces}\label{sec:3}

We begin with the following positive result. 


\begin{theorem}\label{R:1sec3} Let $K$ be a nonempty closed convex subset of a reflexive Banach space $X$. Assume that $T\colon K\to K$ is weakly sequentially continuous and $\alpha$-H\"older Lipschitz with $\alpha\in (0,1)$. Then $T$ has a fixed point if and only if there exists $u\in K$ such that $\{ T^n u\}_{n=1}^\infty$ is bounded.
\end{theorem}

\begin{proof} The orbit $\{ T^n u\}_{n=1}^\infty$ is obviously bounded if $u\in \digamma(T)$, so we only need to prove the converse direction. For this we will argue as in \cite[Lemma 2.3]{KaKi}. Let $L>0$ be as in Definition \ref{def:1}-(3). Next choose $R\geq 1$ so that $\{ T^n(u)\}_{n=1}^\infty\subset B(u;R)$ and $L R^\alpha\leq R$, where for $p\in X$ $B(p;R)$ denotes the closed ball in $X$ of center $p$ and radius $R$. Set
\[
W= K\cap \Bigg( \bigcup_{k=1}^\infty \bigcap_{n=k}^\infty B(T^n (u);R)\Bigg).
\]
Clearly $u\in W$, so $W\neq \emptyset$. If $x\in W$, there is $k\in \mathbb{N}$ so that $\| x - T^n(u)\| \leq R$ for all $n\geq k$. Thus $\| T(x) - T^n(u)\| \leq L \| x - T^{n-1}(u)\|^\alpha\leq L R^\alpha\leq R$ for all $n\geq k+1$, and hence $T( W)\subseteq W$. As $W$ is of course convex, $C:=\overline{W}$ is bounded, closed convex and $T$-invariant. Consequently, $T$ has a fixed point by a result of Arino, Gautier and Penot \cite[Theorem 1]{AGP}. 
\end{proof}


\smallskip 


It must be stressed that the reflexivity in Theorem \ref{R:3sec3} is essential, which is evident from what follows.

\begin{example}\label{R:2sec3} This is adapted from an example by S. Prus (see \cite[p.373]{KhLS}). Pick any $\alpha\in (0,1)$. Let $\mathfrak{L}$ denote a Banach limit on $\ell_\infty$ and define for $x=(t_i)_{i=1}^\infty\in \ell_\infty$,
\[
T(x)= \big( \big|1 - |\mathfrak{L}(x)|^\alpha\big|, |t_1|^\alpha, |t_2|^\alpha, \dots ).
\]
Then the orbit of $0$ is bounded since $T^n(0,0,\dots)=(1, \dots, 1, 0,0, \dots)$ with $1$ in the first $n$ coordinates. Further, it can be easily seen that $\digamma(T)=\emptyset$ and $T$ is $\alpha$-H\"older nonexpansive. 
\end{example}

\begin{example}\label{R:3sec3} Alspach \cite{A} displayed an example of a weakly compact convex set $K\subset L_1$ with diameter $2$ and a fixed-point free isometry $T\colon K\to K$. Notice in particular that $T$ is uniformly $\alpha$-H\"older $2^{1-\alpha}$-Lipschitz for all $\alpha\in (0,1)$. In fact, $K$ and $T$ can be easily modified so that $T$ becomes uniformly $\alpha$-H\"older nonexpansive (cf. \cite[Example 4.1]{Sim})
\end{example}

\smallskip 

Let $X$ be a Banach space, $K\subset X$ and a mapping $T\colon K\to X$. For $\delta>0$, let $\digamma_{\hskip -.05cm\delta}(T) = \big\{ x\in K \colon \| x - T(x) \| \leq \delta\big\}$ denote the $\delta$-approximate fixed point set of $T$. Of course, $\digamma_{\hskip -.05cm 0}(T)=\digamma(T)$. The next result provides a sufficient condition for $\alpha$-H\"older nonexpansive maps on unbounded domains to have bounded orbits. 

\medskip 

\begin{proposition}\label{R:4sec3} Let $K$ be a nonempty subset of a Banach space $X$. Assume that $T\colon K\to K$ is $\alpha$-H\"older nonexpansive with $\alpha\in (0,1)$ and that $\digamma_{\hskip -.05cm\delta}(T) $ is nonempty and bounded for some $\delta\geq 1$. Then there exists $x\in K$ so that $\{ T^n x\}_{n=1}^\infty$ is bounded. 
\end{proposition}

\begin{proof} Since $T$ is $\alpha$-H\"older nonexpansive and $\delta\geq 1$, for $x\in \digamma_{\hskip -.05cm\delta}(T) $ we have
\[
\| T(x) - T^2(x) \| \leq \| x - T(x)\|^\alpha\leq  \delta^\alpha\leq \delta,
\]
so $T\big(\digamma_{\hskip -.05cm\delta}(T) \big) \subseteq\digamma_{\hskip -.05cm\delta}(T) $ and hence $\{ T^n x\}_{n=1}^\infty$ is a bounded for $x\in \digamma_{\hskip -.05cm\delta}(T) $. 
\end{proof}

\medskip 

\begin{remark}\label{R:5sec3} The previous argument was already used in \cite[Lemma 2.2]{KaKi} for the case $\alpha=1$. We also remark that a small modification of the proof in \cite[Lemma 2.1]{KaKi} shows that $\digamma_{\hskip -.05cm\delta}(T) $ is bounded for all $\delta\geq 1$ whenever there exists a point $x_0\in K$ such that
\[
\limsup_{x\in K, \| x\| \to \infty}\frac{\| T(x) - T(x_0)\|}{\| x - x_0\|^\alpha}<1.
\]
\end{remark}

\smallskip 
\begin{remark} Recall a nonlinear mapping $T\colon X\to X$ is called compact provided $T$ maps bounded sequences into sequences that have a convergent subsequence. 
\end{remark}

\smallskip 

\begin{proposition}\label{R:Asec3} Let $K$ be a closed convex subset of a Banach space $X$. Assume that $T\colon K\to K$ is $\alpha$-H\"older $L$-Lipschitz and compact. Then $T$ has a fixed point. 
\end{proposition}

\begin{proof} We may assume that $K$ is unbounded. Now observe that that the set
\[
\big\{ x\in K \colon \, x=\lambda T(u) \text{ for some } 0< \lambda <1\big\}
\]
is bounded, so by the Schaefer fixed point method \cite[p. 504]{Evans} the result follows. 
\end{proof}


\smallskip 

Following the terminology of A. Naor \cite{Na}, for a closed convex set $C\subset X$ we denote by $\mathcal{A}_1(C,X)$ the set of all $\alpha>0$ such that for all bounded convex subset $K\subset C$ and for all $\alpha$-H\"older nonexpansive mapping $F\colon K\to K$ there is $T\colon C\to K$ which is $\alpha$-H\"older nonexpansive and the restriction of $T$ to $K$ is $F$. Recall that for a set $K\in \mathcal{B}(X)$ every affine mapping $T\colon K\to K$ satisfies $\mathrm{d}(T, K)=0$. 

\medskip 


\begin{theorem}\label{R:6sec3}  Let $X$ be a Banach space. The following statements hold:
\begin{itemize}
\item[\it{(1)}]  If $K\in \mathcal{B}(X)$ has the $\alpha$-HFPP for some $\alpha\in (0,1)$ then $K$ has the FPP. 
\item[\it{(2)}]  If $X$ is not reflexive then for every $L> 1$, $\lambda\in (1/L, 1]$ and $\alpha \in (0,1)$ there exist $K\in \mathcal{B}(B_X)$ and a fixed-point free affine mapping $T\colon K\to K$ such that
\[
L^{-1}\| x - y\| \leq \| T(x) - T(y)\| \leq \lambda \| x - y\|^\alpha\quad\text{for all } x, y\in K. 
\]
\item[\it{(3)}] If $X$ has a normalized basic sequence which $1$-dominates some subsequence then for any $\alpha\in (0,1)$ there exist a set $K\in \mathcal{B}(B_X)$ and a $\alpha$-H\"older nonexpansive mapping $T\colon K \to K$ such that $\digamma(T)=\emptyset$ and $\mathrm{d}(T, K)\leq (1/2)^{\frac{2-\alpha}{1-\alpha}}$. 
\item[\it{(4)}]  If $X$ contains a normalized $1$-subspreading basis $\{ e_i\}_{i=1}^\infty$ and has the FPP  then $1 \not\in \mathcal{A}_1(B_X,X)$.
\end{itemize}
\end{theorem}

\begin{proof} (1) Assume for a contradiction that $K$ fails the FPP. Let $T\colon K\to K$ be a nonexpansive mapping without fixed points. With no loss of generality assume that $0\in K$. Set $d=\diam(K)$. Clearly if $d\leq 1$ then $T$ is $\alpha$-H\"older nonexpansive. Assume $d>1$. Since $K$ is convex, $d^{-1}T(x)\in K$ for all $x\in K$. Therefore the mapping $P\colon K\to K$ given by $P(x) = d^{-1}Tx$ is $\alpha$-H\"older nonexpansive and fixed point free. So in both cases we have a contradiction.  

\smallskip 

 (2) The proof is just a modification of the result in \cite{BF}. Indeed, it suffices to take a wide-$(s)$ sequence $(x_n)_{n=1}^\infty$ in $B_X$, consider the closed convex subset
 \[
K=\Bigg\{ x=\sum_{n=1}^\infty t_n x_n\in [x_i] \colon t_n\geq 0\,\,\text{ and }\,\, \sum_{n=1}^\infty t_n=\frac{1}{2}\Big(\frac{\lambda}{L}\Big)^{\frac{1}{1-\alpha}}\Bigg\}
\]
and, for any null sequence $(\alpha_n)_n$ in $(0,1)$, define an affine mapping $T\colon K\to K$ by
\[
T\Bigg( \sum_{n=1}^\infty t_n x_n\Bigg) = (1 - \alpha_1)t_1 x_1  + \sum_{n=2}^\infty \big( (1 - \alpha_n) t_n + \alpha_{n-1} t_{n-1}\big) x_n.
\]

\smallskip 

(3) Let $(x_n)_{n=1}^\infty\subset X$ be a basic sequence such that for some increasing sequence $(n_i)_{i=1}^\infty$ in $\mathbb{N}$, $(x_n)_{n=1}^\infty$ $1$-dominates $(x_{n_i})_{i=1}^\infty$. Pick $\lambda \in (0,1)$ so that $(2\lambda)^{1-\alpha} 2^{2-\alpha}\leq 1$ and take
\[
K=\Bigg\{ x=\sum_{n=1}^\infty t_n x_n\in [x_i] \colon \| x \| \leq \lambda\Bigg\}.
\]
Clearly $K\in \mathcal{B}_c(X)$. Let $T\colon K\to X$ be defined by
\[
T(x) = ( \lambda - \| x \| ) x_1 + \sum_{i=1}^\infty t_i x_{n_i}.
\]
Then $T(K)\subset K$, $\digamma(T)=\emptyset$ and $\mathrm{d}(T, K)\leq (1/2)^{\frac{2-\alpha}{1-\alpha}}$. Moreover, for any $x, y\in K$, 
\[
\| T(x) - T(y)\| \leq 2\| x  - y\|  \leq 2(2\lambda)^{1-\alpha}\| x - y\|^\alpha.
\]

\smallskip 

(4) Suppose that $1 \in \mathcal{A}_1(B_X, X)$. Take $K\subset B_X$ to be the set
\[
K=\Bigg\{ x\in B_X \colon \text{each } e^*_i(x)\geq 0\text{ and } \sum_{i=1}^\infty e^*_i(x)=1\Bigg\},
\]
where $\{ e^*_i\}_{i=1}^\infty$ denotes the biorthogonal functionals associated to the basis. Next define $F\colon K\to K$ to be the right shift map with respect to $\{e_i\}_{i=1}^\infty$.  It follows that $\digamma(F)=\emptyset$. Also the $1$-subspreading assumption on the basis implies $F$ is nonepansive. Therefore since $1 \in \mathcal{A}_1(B_X, X)$ there is a nonexpansive extension  $T\colon B_X\to K$ of $F$ to $B_X$. It follows that $\digamma(T)=\emptyset$, contrary to our assumption that $X$ has the FPP. 
\end{proof}

\smallskip 


\begin{corollary}\label{R:7sec3} Let $X$ be a Banach space. The following are equivalent:
\begin{itemize}
\item[\it{(1)}] $X$ is reflexive.
\item[\it{(2)}] $X$ has the FPP for $\alpha$-H\"older nonexpansive affine maps.
\end{itemize}
\end{corollary}

\medskip 

Our last result in this section concerns hyperconvex metric spaces. 

\smallskip 


\begin{proposition}\label{R:8sec3} Let $(K, \rho)$ be a hyperconvex metric space with $\diam(K)\leq 2$. Assume that $T\colon K\to K$ is fixed-point free and $\alpha$-H\"older nonexpansive. Then $\alpha<1$. 
\end{proposition}

\begin{proof} Assume the contrary that $\alpha\geq 1$. Since hyperconvex spaces have the FPP (see \cite{So}) and $\digamma(T)=\emptyset$, $\alpha>1$. Now we follow \cite{So} to reach to a contradiction. For a nonempty set $C\subset K$ and $x\in K$, set $r_x(C)=\sup\{ \rho(x,y) \colon y\in C\}$. Set $\mathcal{M} = \big\{ D\in \mathcal{A}(K) \colon D\neq \emptyset,\,\, T(D)\subseteq D\big\}$, where $\mathcal{A}(K)$ is the family of all admissible subsets of $K$. Since $K\in \mathcal{M}$, $\mathcal{M}\neq \emptyset$. By Zorn's lemma, $\mathcal{M}$ has a minimal element $D$. Then $D$ is a single set. Indeed, assume that $\diam(D)>0$. Since $T(D)\subset D$, one readily has $\mathrm{cov}(T(D))\subseteq D$ (cf. \cite[p. 85]{KhKi}). Hence 
\[
T(\mathrm{cov}(T(D)))\subseteq T(D)\subseteq \mathrm{cov}(T(D)). 
\]
As $\mathrm{cov}(T(D))\in \mathcal{M}$ and $D$ is minimal, $D=\mathrm{cov}(T(D))$. In consequence we have $r_x(D)= r_x(T(D))$ for all $x\in K$. Let $d=\diam(D)$ and set
\[
M=D\cap \bigcap_{x\in D}B(x ; d/2).
\]
Thus $M$ is nonempty. Now observe that $T(M)\subseteq M$. Indeed, fix $z\in M$. Then since $T$ is $\alpha$-H\"older nonexpansive, we have $r_{T(z)}(T(D))\leq r^\alpha_z(D)$. Notice also that $r_z(D)=d/2$, so $r_{T(z)}(D)= r_{T(z)}(T(D))\leq r_z^\alpha(D) \leq d/2$ since $d\leq 2$ and $\alpha>1$. Using that $T(z)\in D$, we get $T(z)\in M$. It turns out that $D=M$ by the minimality of $D$. But $\diam(M)=d/2$, a contradiction.
\end{proof}


\section{H\"older contractive maps on unit balls}\label{sec:4}

In this section we collect some facts related to H\"older maps on unit balls. In what follows we shall use $B_X(r)$ (resp. $S_X(r)$) to denote the closed ball (sphere) centered at the origin of $X$ with radius $r>0$. 

\medskip 


\begin{proposition}\label{R:1sec4} Let $X$ be an infinite dimensional Banach space. For $\lambda, \alpha\in (0,1)$ the following hold:
\begin{itemize}
\item[\it{(1)}] There exists a $\alpha$-H\"older $\lambda$-contractive mapping $T\colon B_X\to B_X$ with $\digamma(T)=\emptyset$ and $\mathrm{d}(T, B_X)>0$.
\item[\it{(2)}] For any $e\in S_X$ there is a weakly continuous $\alpha$-H\"older nonexpansive mapping $T\colon B_X \to B_X$ which is $\alpha 2^{1-\alpha}$-Lipschitz and $\digamma(T)=\{e\}$. 
\item[\it{(3)}] If $T\colon B_X \to B_X$ is $\alpha$-H\"older nonexpansive then $T$ is not onto.
\item[\it{(4)}] If $K\subset B_X$ is convex and $T\colon K\to K$ is nonexpansive then for any $\varepsilon\in (0,1)$ there is a $\alpha$-H\"older $(\varepsilon + \diam(K)^{1-\alpha})$-Lipschitz map $T_\varepsilon \colon K\to K$ such that $\digamma(T_\varepsilon) = \digamma(T)$, $\mathrm{d}(T_\varepsilon, K)\leq \mathrm{d}(T, K)$ and $\| T(x) - T_\varepsilon(x)\| \leq \varepsilon$ for all $x\in K$.
\item[\it{(5)}] If $F\colon B_X\to B_X$ is nonexpansive with $\digamma(F)=\emptyset$, then $B_X$ fails the FPP for uniformly $\alpha$-H\"older $\lambda$-contractive maps with null minimal displacement. 
\item[\it{(6)}] For $X=\mathrm{c}$ or $X=\co$, $B_X$ fails the FPP for uniformly $\alpha$-H\"older $\lambda$-contractive maps with null minimal displacement. 
\item[\it{(7)}] For $1\leq p< \infty$ there exists a convex set $K\subset B_{\ell_p}$ which fails the FPP for uniformly $\alpha$-H\"older $\lambda$-contractive affine maps. 
\item[\it{(8)}] $B_{\ell_1}$ fails the FPP for uniformly $\alpha$-H\"older $\lambda$-contractive mappings with null minimal displacement. 
\item[\it{(9)}] If $X$ contains a complemented isometric copy of a Banach space $E$ in which $B_E$ fails the FPP for (uniformly) $\alpha$-H\"older $\lambda$-contractive mappings with null minimal displacement, then $B_X$ also fails this property. 
\end{itemize}
\end{proposition}

\begin{proof} (1) The proof is contained in the proof of Proposition \ref{prop:1}. 

(2) By the geometric form of Hahn-Banach theorem (see e.g. \cite{FHHMZ}), there is $\varphi\in X^*$ so that $\varphi(e)=1$ and $\varphi(x)<1$ for all $x\in X$ with $\| x \|<1$. One can readily verify that $\| \varphi\| \leq 1$. Then the map $T\colon B_X \to B_X$ given below meets the described properties
\[
T(x) = \Big( \frac{1 + \varphi^2(x)}{2}\Big)^\alpha\cdot e,\quad x\in B_X.
\] 

(3) Suppose $T$ maps onto $B_X$. For $x\in B_X$ pick $y\in B_X$ so that $x=T(y)$. Hence $\| T(0) - x\| =\| T(0) - Ty\| \leq \| y\|^\alpha\leq 1$. So, $\| T(0) - x\| \leq 1$ for all $x\in B_X$. Since $0$ is the only vector in $B_X$ with such property, we have $T(0)=0$. Consequently, $T^{-1}(S_X)\subset S_X$ (cf. \cite[Theorem 2.3]{CKOW}). Fix $x\in T^{-1}(S_X)$ and choose $y\in B_X$ so that $T(y) = - T(x)$. Then 
\[
2=\| T(x) - (-T(x)) \| = \| T(x) - T(y)\| \leq \| x - y\|^\alpha\leq \| x\|^\alpha + \|y\|^\alpha\leq 2. 
\]
This implies $\| x - y\|= 2^{1/\alpha}>2$, a contradiction (since $\alpha<1$ and $\| x - y\| \leq 2$).

\medskip 

(4) For $x\in K$ define $T_\varepsilon\colon  K\to K$ by
\[
T_\varepsilon(x) = \frac{\varepsilon \|x\|^\alpha}{4(1 + \| x\|^\alpha)} x  + \Big( 1- \frac{\varepsilon\| x\|^\alpha}{4(1 + \| x\|^\alpha)}\Big) T(x). 
\]
Direct calculation shows that $T_\varepsilon$ fulfills the stated properties.

\medskip 

(5) Fix $r>0$ so that $2r^{1-\alpha}\leq \lambda$ and set $K=\{ x\in B_X \colon \| x \| \leq r\}$. Next consider the $2$-Lipschitz retraction $R\colon B_X\to K$ given $R(x) = x$ if $\| x\| \leq r$ and $R(x) = rx/\| x\|$ for $r\leq \| x \| \leq 1$. Then the mapping $T\colon B_X \to B_X$ given by $T(x) =r  F(R(x)/r)$ satisfies $\digamma(T)=\emptyset$, $\mathrm{d}(T, B_X)=0$, 
\[
\| T(x) - T(y)\| \leq \| R(x) - R(y)\| \leq \lambda \| x - y\|^\alpha
\]
and, consequently, $T$ is uniformly $\alpha$-H\"older $\lambda$-contractive.

\smallskip 

(6) It suffices to apply  the result (5) for the map $F\colon B_X\to B_X$ given by
\[
F(t_1, t_2, t_3,\dots) = (1, 0, |t_1|, |t_2| , |t_3| ,\dots).
\]

\smallskip 

 (7) Impelled by the proof of Theorem \ref{R:6sec3} we let
\[
K=\Big\{ x=(t_i)_{i=1}^\infty \in \ell_p \colon t_i \geq 0\text{ and } \sum_{i=1}^\infty t_i = \frac{\lambda^{\frac{p}{1-\alpha}}}{2}\Big\}.
\]
Then $K$ is a nonempty convex subset of $B_{\ell_p}$. Notice however that $K$ is closed only if $p=1$. Let us verify the stated properties for $K$. Let $\{e_i\}_{i=1}^\infty$ denote the unit vector basis of $\ell_p$. Consider $F\colon K\to K$ to be the right shift on $\{e_i\}_{i=1}^\infty$. Then clearly $F$ is affine and fixed-point free. Since $K$ is convex, $\mathrm{d}(T, K)=0$. Now for $x=(t_i)_{i=1}^\infty$ and $y=(s_i)_{i=1}^\infty$ in $K$, note that $t_i + s_i\leq \lambda^{\frac{p}{1-\alpha}}<1$. Hence for any $n\in \mathbb{N}$, 
\[
\begin{split}
\| F^n(x) - F^n(y)\|_p& 
\leq \Bigg( \sum_{i=1}^\infty | t_i| +|s_i|\Bigg)^{\frac{1 - \alpha}{p}} \| x - y\|_p^{\alpha}\leq \lambda \| x - y\|_p^{\alpha}.
\end{split}
\]

(8) Let $\theta=\sqrt{\alpha}$ and set
\[
K=\Bigg\{ x=(t_i)_{i=1}^\infty \in \ell_1 \colon t_i \geq 0\text{ and } \sum_{i=1}^\infty t_i  = 4^{-1}\Big(\frac{\lambda}{8^\theta}\Big)^{\frac{1}{1-\theta}}\Bigg\}.
\]
Set $r = 4^{-1}\big(\lambda/8^\theta\big)^{\frac{1}{1-\theta}}$. Note that $K=S_{\ell_1}^{+}(r)$ (the positive face of $S_{\ell_1}(r)$). Let $S$ be the right shift operator on $\ell_1$. Define $F\colon K\to K$ by $F=S|_K$. Also let $G\colon B_{\ell_1} \to B_{\ell_1}(r)$ be a $2$-Lipschitz retraction (see (5)). Now the idea is to take a Lipschitz retraction from $B_{\ell_1}(r)$ onto $K$. A direct way to do this is as in \cite{AC}. Here we will only sketch the arguments.  Let $\{ e_j\}_{j=1}^\infty$ be the unit basis of $\ell_1$. Consider the set $\Gamma=\big\{ x=(t_j)_{j=1}^\infty\in B_{\ell_1} \colon r/2\leq \| x\|_1 < r\big\}$ and define $\iota \colon \Gamma \to \mathbb{N}$ by setting 
\[
\iota(x) = \min\Bigg\{ j\in \mathbb{N} \colon \sum_{k=j+1}^\infty | t_k| < r - \| x\|_1\Bigg\}, \quad x=(t_j)_{j=1}^\infty \in \Gamma.
\]
Also let $\mu \colon \Gamma \to (0,1]$ be such that, for all $x\in \Gamma$, 
\[
\mu(x) | t_{\iota(x)}| + \sum_{k=\iota(x) +1}^\infty | t_k| = r - \| x \|_1,
\]
and, finally, define $Q\colon \overline{\Gamma} \to B_{\ell_1}$ by 
\[
Q(x)= \left\{ 
\begin{aligned}
&\mu(x) t_{\iota(x)} e_{\iota(x)} + \sum_{k=\iota(x) +1}^\infty t_k e_k &\text{ if } \| x\|_1<r,\\
&0     & \text{ if } \|x\|_1=r,
\end{aligned}
\right.
\]
where $\overline{\Gamma}$ is the closure of $\Gamma$. By proceeding exactly as in \cite[pp. 74-75]{AC}, it can readily seen that $Q$ is $3$-Lipschitzian. An $8$-Lipschitz retraction $R_2 \colon B_{\ell_1}(r) \to S_{\ell_1}(r)$ is then defined as follows: 
\[
R_1(x)=\left\{ 
\begin{aligned}
(r - 2\| x\|_1) e_1 + 2S(x)\quad &\text{ if } \| x\|_1 \leq r/2,\\
(I - Q)(x) + 2S\big( Q(x)\big)\quad &\text{ if } r/2\leq \| x\|_1 \leq r.
\end{aligned}
\right.
\]
Finally, consider the retraction $R_1\colon S_{\ell_1}(r) \to K$ given by $R_1((t_j)_{j=1}^\infty ) = ( |t_j| )_{j=1}^\infty$, and define $T\colon B_{\ell_1}\to K$ by $T:= F\circ R_1\circ R_2 \circ G$. Clearly $\digamma(T)= \emptyset$ and $\mathrm{d}(T, B_{\ell_1}) \leq \mathrm{d}(F, K)=0$. Hence $T$ is $\alpha$-H\"older nonexpansive and, since $G$, $R_1$ and $R_2$ are retractions, one also has $\|T^n(x) - T^n(y)\| \leq \| x - y\|^{\alpha}$ for all $x, y\in B_{\ell_1}$.

\smallskip 

(9) Let $Y$ denote the isometric copy of $E$ in $X$. By assumption there exist a (uniformly) $\alpha$-H\"older nonexpansive mapping $F\colon B_Y \to B_Y$ with $\mathrm{d}(F, B_Y)=0$ and $\digamma(F)=\emptyset$,  and a projection $P\colon X\to Y$. Observe that if $r>0$ is such that $r^{1-\alpha} 2^{1- \alpha}\| P\| \leq \lambda$, then the map $F_r \colon B_Y(r) \to B_Y(r)$ given by $F_r(x) = r F(x/r)$ satisfies $\digamma(F_r)=\emptyset$, $\mathrm{d}(F_r, B_Y(r))=0$ and $\| F_r^n(x) - F_r^n(y)\| \leq \lambda \| x - y\|^\alpha$ for all $x, y\in B_Y(r)$ (and every $n\in\mathbb{N}$ in the uniform case). Thus taking a $2$-Lipschitz retraction $R\colon Y\to B_Y(r)$, the mapping $T\colon B_X \to B_X$ defined by $T:= F_r\circ R \circ P$ fulfills the stated properties. 
\end{proof}


\smallskip 

\begin{remark}\label{R:2sec4} It worths to point out that when $\alpha=1/2$ the mapping $T$ given in part {\it{(2)}} above is $\sqrt{2}/2$-contractive and satisfies
\[
T^n(x) =\Bigg( \sum_{i=1}^n \frac{1}{2^i} + \frac{\varphi^2(x)}{2^n}\Bigg)^{1/2}\cdot e,\quad x\in B_X,\, n\in \mathbb{N}. 
\]
\end{remark}

\begin{remark}\label{R:3sec4} Let $X\in\{\mathrm{c},\,\co\}$. As in {\it{(7)}}, case $p=1$, one can find $K\in \mathcal{B}(B_X)$ and a uniformly $\alpha$-H\"older $\lambda$-contractive affine mapping $T\colon K \to K$ with $\digamma(T)=\emptyset$. For example, take $K=B^+_X$ and pick a sequence $(\alpha_n)_{n=1}^\infty\subset (0, 1)$ with $\lim_n \alpha_n=1$. Set $\beta_n:=1 - \alpha_n$. Recall the Schauder basis for $\mathrm{c}$ is $\{ e_n\}_{n=0}^\infty$ with $e_0:=(1,1,\dots)$ and, for $n\geq 1$, $e_n$ is the $n$-\textrm{th} canonical unit vector in $\co$. Hence, any $x=(t_j)_{j=1}^\infty\in \mathrm{c}$ has the following basis expansion: $x = \sum_{j=0}^\infty e^*_j(x) e_j$ with $e^*_0(x)= \lim_{j\to \infty} t_j$ and $e^*_j(x)= t_j$,\, $j\in \mathbb{N}$. Now define $T\colon B_X^+\to B_X^+$ by
\[
T(x)= \sum_{n=1}^\infty \alpha_n e^*_n(x) e_n  + \sum_{n=1}^\infty \beta_n e_n.
\]
It is easily seen that $\digamma(T)=\emptyset$ and $\mathrm{d}(T, B_X^+)=0$. Moreover, we have 
\[
\| T^n(x) - T^n(y)\|_\infty \leq \| x - y\|_\infty \leq \| x - y\|_\infty^\alpha\quad\text{for all }x, y\in B_X^+.
\]
\end{remark}


\begin{remark} Examples of spaces $X$ as in (9) include non-polyhedral Lindenstrauss spaces (cf. \cite[Proposition 3.1]{CMP} and \cite[Theorem 4.3]{CMPV}). Recall a {\it Lindenstrauss space} is a real Banach space $X$ for which $X^*=L_1(\mu)$ for some measure $\mu$. 
\end{remark}

\smallskip 

\begin{remark}\label{R:5sec4}  Goebel and Kirk, in \cite{GK0},  have built an example of an asymptotically nonexpansive mapping $T\colon B_{\ell_2}\to B_{\ell_2}$ which is not nonexpansive and and has $0$ as a unique fixed point. More precisely, letting $A_i = (1-1/i^2)$ and $\kappa_n = 2\prod_{i=2}^n A_i$ one has $\kappa_n\to 1$ and $\| T^n(x) - T^n(y)\|_2 \leq \kappa_n \| x - y\|_2$ for all $x, y\in B_{\ell_2}$, $n\in \mathbb{N}$. Here $\| \cdot\|_2$ denotes the $\ell_2$-norm. It is clear that this map satisfies
\[
\| T^n(x) - T^n(y)\| \leq \kappa_n 2^{1-\alpha} \| x - y\|^\alpha, \quad\forall x, y\in B_{\ell_2},\, \alpha\in (0,1),\, n\in \mathbb{N}.
\]
Notice however that it is not possible, in principle, to suppress the factor $2^{1-\alpha}$ to get an assymptotically $\alpha$-H\"older nonexpansive map which is not Lipschitz. 
\end{remark}

\smallskip 

\begin{proposition}\label{R:6sec4} For any $\varepsilon, \alpha\in (0,1)$ there exists an asymptotically $\alpha$-H\"older nonexpansive mapping $T\colon B_{\ell_2} \to B_{\ell_2}$ which is not Lipschitzian and $T(0)=0$. 
\end{proposition}

\begin{proof} Let $(A_i)_{i=1}^\infty$ be as above. Define a mapping $F\colon B_{\ell_2}^+ \to B_{\ell_2}^+$ by 
\[
F(t_1, t_2, \dots)= (0, t_1^\alpha, A_2 t_2, A_3t_3, \dots).
\]
Then $F$ is $\alpha$-H\"older $2$-Lipschitz and $F(0)=0$. Moreover, direct calculation shows that $F$ is asymptotically $\alpha$-H\"older nonexpansive with sequence $\{ 2\prod_{i=2}^n A_i\}_n$. Let us consider the nonexpansive retraction $R\colon B_{\ell_2}\to B_{\ell_2}^+$ given by $R\big( (t_n)_{n=1}^\infty\big) = ( t^+_n)_{n=1}^\infty$. Now, for $x=(t_n)_{n=1}^\infty\in B_{\ell_2}$, set $T(x)= F(R(x))$. It is easy to see that $T$ is asymptotically $\alpha$-H\"older nonexpansive, $\digamma(T)=\{0\}$ and is not Lipschitzian. 
\end{proof}

\vskip .2cm
We now proceed to solve partially issue ($\mathcal{Q}3$) for $X=\ell_2$. We first prove the following:

\begin{proposition} Let $X$ be a Banach space. Assume that $B_X$ fails the FPP for Lipschitz maps with null minimal displacement. Then for any $\alpha\in (0,1)$ and $\lambda>0$ there exists a fixed-point free $\alpha$-H\"older $\lambda$-Lipschitz mapping $T\colon B_X\to B_X$ which has null minimal displacement.
\end{proposition}

\begin{proof} By assumption there is a $L$-Lipschitz mapping $S\colon B_X \to B_X$ such that $\digamma(S)=\emptyset$ and $\mathrm{d}(S, B_X)=0$. Choose $r>0$ so that $2Lr^{1-\alpha}\leq \lambda$. Next define $S_r\colon B_X(r) \to B_X(r)$ by $S_r(x) = r S(x/r)$. Then $\digamma(S_r)=\emptyset$ and $\mathrm{d}(S_r, B_X(r))\leq r \mathrm{d}(S, B_X)=0$. In addition, for every $x, y\in B_X(r)$ we have
\[
\| S_r(x) - S_r(y)\|\leq L (2r)^{1-\alpha}\| x - y\|^\alpha.
\]
Finally, take $R\colon B_X\to B_X(r)$ to be a $2$-Lipschitz retraction and define $T\colon B_X\to B_X$ by $T(x) = S_r(Rx)$. Note that $\digamma(T)=\emptyset$ and $\mathrm{d}(T, B_X)=0$. Moreover, 
\[
\|T(x) - T(y)\|\leq L (2r)^{1-\alpha} \| Rx - Ry\|^\alpha\leq \lambda \| x - y\|^\alpha.
\]
\end{proof}

This result combined with the next one provides the mentioned solution.

\begin{proposition}\label{R:8sec4} Let $H$ be an infinite dimensional Hilbert space. There exists a Lipschitz mapping $T\colon B_{H}\to B_{H}$ which has no fixed points and has null minimal displacement. 
\end{proposition}

\begin{proof} It is well-known (and simple to check) that $H$ contains a linear complemented isometric copy of $\ell_2$. By Proposition \ref{R:1sec4}-(9), we may assume that $H=\ell_2$. Let $(e_i)_{i=1}^\infty$ denote the unit basis of $\ell_2$ and set
\[
K = \Bigg\{ \sum_{i=1}^\infty t_i e_i \colon \, t_1\geq t_2 \geq t_3\geq \dots \geq 0\,\text{ and }\, \sum_{i=1}^\infty t^2_i\leq 1\Bigg\}.
\]
It is clear that $K\in \mathcal{B}(B_{\ell_2})$. By a result of P.K. Lin \cite{Lin3} there exists a uniformly asymptotically regular Lipschitz mapping $F\colon K\to K$ without fixed points. The asymptotic regularity of $F$ implies $\mathrm{d}(F, K)=0$. Now take a nonexpansive projection $P\colon \ell_2\to K$ and consider the mapping $T\colon B_{\ell_2}\to K$ given by $T(x) = F(Px)$. Then $T$ fulfills the properties mentioned in the statement of the proposition.  
\end{proof}

\begin{remark} T. Gallagher, C. Lennard and R. Popescu \cite{GLP} have exhibited a class of non-weakly compact sets $K\in \mathcal{B}(\mathrm{c},\| \cdot\|_\infty)$ that enjoy the FPP for nonexpansive mappings. It is not hard to show that the same is not true for the $\alpha$-HFPP. 
\end{remark}

\smallskip 

\begin{proposition}\label{R:7sec4} For any $\alpha\in (0,1)$, $B_{\ell_\infty}$ fails the FPP for uniformly $\alpha$-H\"older nonexpansive mappings having null minimal displacement.
\end{proposition}

\begin{proof} By \cite[Theorem]{AP} we known that hyperconvex metric subspaces of $\ell_\infty$ are absolute nonexpansive retracts. So, all we need to do is to consider a hyperconvex subset $K$ of $B_{\ell_\infty}$ and built a $\alpha$-H\"older nonexpansive mapping $F\colon K\to K$ such that $\digamma(F)=\emptyset$ and $\mathrm{d}(F, K)=0$. Choose $N\in \mathbb{N}$ so that $2\leq N^\alpha$. Set
\[
K=\Big\{ x=(t_n)_{n=1}^\infty \in \mathrm{c} \colon 0\leq t_n \leq 1/N\text{ for all } n\geq 1\Big\}.
\]
Clearly $K$ is closed convex. Moreover, as an interval in $\ell_\infty$, it is hyperconvex (see \cite[Remark 4.1, p.49]
{GK}). We now define a mapping $F\colon K\to K$ as follows: for $x= (t_n)_{n=1}^\infty\in K$, put
\[
F(x)= \big(1/N, t_2 t_1^\alpha, t_1, t_2, t_3, \dots\big).
\]
It is easy to see that that $T$ is well-defined and $\digamma(F)=\emptyset$. Let us prove that $F$ is uniformly $\alpha$-H\"older nonexpansive. For $n\geq 1$, an easy inductive procedure yields
\[
F^{n+1}(x) = \Big(\frac{1}{N}, t_2\Big( \frac{t_1}{N^n}\Big)^\alpha, \frac{1}{N}, t_2\Big( \frac{t_1}{N^{n-1}}\Big)^\alpha,  \dots, \frac{1}{N}, t_2\Big( \frac{t_1}{N}\Big)^\alpha, \frac{1}{N}, t_2 t_1^\alpha, t_1, t_2, t_3, \dots\Big).
\]
Then, since $2\leq N^\alpha$, for every $x=(t_i)_{i=1}^\infty,\, y= (s_i)_{i=1}^\infty\in K$
\[
\begin{split}
\| F^n(x) - F^n(y)\|_\infty &\leq \sup_{i\in \mathbb{N},\, 1\leq j\leq n}\Big\{ \Big| t_2\Big( \frac{ t_1}{N^{n-j}}\Big)^\alpha - s_2\Big( \frac{s_1}{N^{n-j}}\Big)^\alpha\Big|, |t_i - s_i|\Big\}\\[1.2mm]
&\leq \sup_{i\in \mathbb{N}}\Big\{ (t_1^\alpha + s_2)\| x - y\|^\alpha_\infty, |t_i - s_i| \Big\}\\[1.2mm]
&\leq \| x - y\|^\alpha_\infty.
\end{split}
\]
It remains to show that $\mathrm{d}(F, K)=0$. To this end, we will consider a scaling argument. For $\lambda\in (0,1)$ and $x\in K$, set $F_\lambda(x) = F(\lambda x)$. Note that 
\[
\begin{split}
F^n_\lambda(x) = \Bigg(\frac{\lambda^0}{N}, \lambda^{n+n\alpha}t_2\Big( \frac{t_1}{N^{n-1}}\Big)^\alpha, &\frac{\lambda}{N}, \lambda^{n +(n-1)\alpha} t_2 \Big( \frac{t_1}{N^{n-2}}\Big)^\alpha, \dots\dots
 \\[1.2mm]
&\hskip .5cm\dots\dots, \frac{\lambda^{n-1}}{N}, \lambda^{n + \alpha} t_2 t_1^\alpha, \lambda^n t_1, \lambda^n t_2, \dots, \Bigg),
\end{split}
\]
where, as usual, $F^n_\lambda$ means the $n$-\textrm{th} iterate of $F_\lambda$. Therefore, after some computation,
\[
\| F^n_\lambda (x) - F^{n+1}_\lambda(x)\|_\infty \leq \lambda^n,\quad n\in \mathbb{N}.
\]
Consequently, $\mathrm{d}(F_\lambda, K) \leq \lambda^n$ for every $n\in \mathbb{N}$. Then $\mathrm{d}(F_\lambda, K)=0$. This certainly implies $\mathrm{d}(F, K)=0$. Indeed, fix $\varepsilon>0$. For some $x\in K$ we get $\| x - F_\lambda(x)\|_\infty\leq \varepsilon$, and this in turn implies
\[
\begin{split}
\mathrm{d}(F, K)& \leq \| x - F(x) \| \leq \varepsilon + \| F(x) - F(\lambda x)\|\\[1.1mm]
&\leq \varepsilon + (1- \lambda)^\alpha \| x\|^\alpha\\[1.1mm]
&\leq \varepsilon + (1- \lambda)^\alpha.
\end{split}
\]
Letting $\varepsilon\to 0$ and $\lambda \to 1$, we have $\mathrm{d}(F, K)=0$. This concludes the proof. 
\end{proof}


\section{H\"older FPP on unit balls of spaces containing $\co$ or $\ell_1$}\label{sec:5}

Theorem \ref{R:6sec3}, part (2), shows that if $X$ is any non-reflexive Banach space then some $K\in \mathcal{B}(B_X)$ fails the FPP for $\alpha$-H\"older nonexpansive mappings with null minimal displacement. But, in general, no retraction from $B_X$ to $K$ is available and, moreover, the  affine mapping obtained there is not necessarily uniformly $\alpha$-H\"older nonexpansive. In this section we show that this difficulty can be circumvented in the framework of Banach spaces containing isomorphic copies of  $\co$ or $\ell_1$. 

\medskip 

The first main result of this section is the following. 


\begin{theorem}\label{R:1sec5} Let $X$ be a Banach space containing a copy of $\co$. Then:
\begin{itemize}
\item[\it{(1)}] There exist a set $K\in \mathcal{B}(B_X)$ and a family of $\alpha$-H\"older nonexpansive maps $\{T_\alpha\}_{\alpha\in (0,1)}$ on $K$ such that $T(x)= \lim_{\alpha\to 1^-}T_\alpha(x)$ exists for all $x\in K$, and $\mathrm{d}(T_\alpha, K)\leq (1-\alpha)\displaystyle\sup_{0< t<1}t^\alpha |\ln t|$. In addition each $\digamma(T_\alpha)=\emptyset$, and either $\digamma(T)=\emptyset$ or $| \digamma(T)|=1$. \vskip .1cm
\item[\it{(2)}] For every $\alpha \in (0, 1)$ there exist $K\in \mathcal{B}(X)$ and a fixed point free  uniformly $\alpha$-H\"older $\lambda$-contractive mapping $F\colon K\to K$ satisfying $\mathrm{d}(F, K)=0$. \vskip .1cm 
\item[\it{(3)}] If $X$ is a separably Sobczyk space then for every $\alpha, \lambda \in (0, 1)$ there exists a fixed point free mapping $T\colon B_X\to B_X$ which is uniformly $\alpha$-H\"older $\lambda$-contractive and satisfies $\mathrm{d}(T, B_X)=0$.
\end{itemize}
\end{theorem}

\smallskip 

\begin{proof}

(1) Fix $\delta\in (0,1)$. Pick any null sequence $(\sigma_i)$ in $(0,1)$ so that $\sigma_1< 1-\delta$ and
\begin{equation}\label{eqn:1sec2}
\sigma_{i+1} \leq  (1 - \delta)\sigma_i,\quad i\in \mathbb{N}.
\end{equation}
Take $(x_n)_{n=1}^\infty $ to be as in (\ref{eqn:B}) with $\delta_1=\delta$. Set
\[
K=\Bigg\{ \sum_{i=1}^\infty t_i x_i \colon\, \text{each } \sigma_i \leq t_i \leq t_1= 1- \delta\,\Bigg\}.
\]
Notice $K$ clearly belongs to $\mathcal{B}(B_X)$. For $\alpha\in (0,1)$, define $T_\alpha\colon K\to  [x_i]$ by 
\[
T_\alpha\Bigg(  \sum_{i=1}^\infty t_i x_i\Bigg) = (1-\delta)x_1 + \sum_{i=1}^\infty (1-\delta )t^\alpha_i x_{i+1}. 
\]
Let us prove that $T_\alpha(K)\subseteq K$. Fix  $x=\sum_{i=1}^\infty t_i x_i\in K$. Set $\tilde{t}_i= 1-\delta$ if $i=1$ and $\tilde{t}_i = (1-\delta)t^\alpha_{i-1}$ for $i\geq 2$. So $T_\alpha(x) =\sum_{i=1}^\infty \tilde{t}_i x_i$ and for each $i\in \mathbb{N}$, we have $\sigma_i \leq t_i \leq 1- \delta$. Fix $i\in \mathbb{N}$. It follows that $\sigma_1\leq \tilde{t}_1\leq 1-\delta$ and, for $i\geq 1$, $\tilde{t}_{i+1}= (1 - \delta) t^\alpha_i \leq (1 -\delta)^{1+ \alpha}\leq 1 - \delta$ and
\[
\tilde{t}_{i+1}=(1- \delta) t_i^\alpha \geq (1 - \delta) t_i \geq (1 - \delta)\sigma_i = \sigma_{i+1}\qquad \big(\text{by } (\ref{eqn:1sec2})\big).
\]
Thus $T_\alpha(x)\in K$, and $K$ is $T_\alpha$-invariant. 

We now proceed to show that $\digamma(T_\alpha)=\emptyset$. Let $\alpha$ be fixed and assume towards a contradiction that $x = T_\alpha (x)$ for some $x=\sum_{i=1}^\infty t_i x_i\in K$. Since $(x_n)_{n=1}^\infty$ is basic, 
\[
t_1=1-\delta\quad\text{ and } \,\, t_{i+1} = (1 - \delta)t^\alpha_i\quad\text{for all } i\in \mathbb{N}. 
\]
So, $t_{n+1}= (1-\delta)^{\sum_{i=0}^n \alpha^i}$ for all $n\in \mathbb{N}$ which implies $\lim_{n\to \infty} t_{n+1} = (1-\delta)^{\frac{1}{1-\alpha}}$, contradicting that $(t_i)_{i=1}^\infty \in \co$. Hence each $\digamma(T_\alpha)=\emptyset$. 

Next we shall show that $T_\alpha$ is $\alpha$-H\"older nonexpansive and that  
\begin{equation}\label{eqn:2sec2}
\mathrm{d}(T_\alpha, K) \leq (1 - \alpha)\sup_{0< t<1}t^\alpha |\ln t|. 
\end{equation}
Fix any points $x, y\in K$, say $x=\sum_{i=1}^\infty t_i x_i$ and $y=\sum_{i=1}^\infty s_i x_i$. For $i\in \mathbb{N}$, we write $a_i = t_i -s_i$. Then from the right hand side in (\ref{eqn:B})  we deduce
\[
\| T_\alpha (x)  - T_\alpha (y)\| \leq  (1- \delta )\sup_{i\in\mathbb{N}}\big| t^\alpha_i - s^\alpha_i\big|\leq  (1- \delta)\sup_{i\in\mathbb{N}} | t_i - s_i|^\alpha,
\]
where in the last inequality we used that $t\mapsto t^\alpha$ is H\"older nonexpansive. Hence the left hand side of (\ref{eqn:B}) implies
\[
\begin{split}
\| T_\alpha(x)  - T_\alpha(y)\|^{1/\alpha} &\leq (1 - \delta)^{1/\alpha} \sup_{i\in \mathbb{N}}|a_i|
\leq \Bigg\| \sum_{i=1}^\infty a_i x_i\Bigg\| =\| x - y\|. 
\end{split}
\]
This shows that $T_\alpha$ is $\alpha$-H\"older nonexpansive.

\smallskip 

\noindent {\bf Claim.} There exists $\theta\in (0,1)$ such that $\{ T_\alpha\}_{\theta\leq \alpha<1}$ pointwisely converges to a nonexpansive map $T\colon K\to K$ as $\alpha \to 1$. Indeed, pick a number $\theta\in (0,1)$ so that $t^{\theta} \cdot | \ln t| \leq 1/2$ for all $t\in (0,1]$. Fix $x=\sum_{i=1}^\infty t_i x_i$ in $K$. For any $\alpha, \beta \in (0,1)$ 
\[
T_\alpha(x) - T_\beta(x) =(1-\delta) \sum_{i=1}^\infty \big( t^\alpha_i -t^\beta_i\big) x_{i+1},
\]
and then by (\ref{eqn:B}) we have
\[
\| T_\alpha(x) - T_\beta(x) \| \leq (1-\delta)\sup_{i\in \mathbb{N}}| t^\alpha_i - t^\beta_i |.
\]
Fix $i\in \mathbb{N}$. It is easily seen that
\begin{equation}\label{eqn:3sec6}
| t^\alpha_i - t^\beta_i| \leq t^{\min(\alpha, \beta)}_i \cdot | \ln t_i| \cdot |\alpha -\beta|.
\end{equation}
Therefore for any $\theta\leq \alpha, \beta<1$, $\| T_\alpha(x) - T_\beta(x) \| \leq \frac{1-\delta}{2}| \alpha - \beta|$. This shows $\{ T_\alpha(x)\}_{\theta\leq \alpha<1}$ is Cauchy. 

Fix a sequence $(\alpha_n)_n$  in $(0, 1)$ converging to $1$. We may assume that $\alpha_n\geq \theta$ for all $n$. Since $K$ is closed and $T_{\alpha_n}(x)\in K$, we have $T(x) = \lim_{n\to \infty} T_{\alpha_n}(x)$ for some $T(x)\in K$. Consider the map $x\mapsto T(x)$. Clearly $T\colon K\to K$ is nonexpansive for
\[
\| T_{\alpha_n}(x) - T_{\alpha_n}(y) \| \leq \| x  - y\|^{\alpha_n},\quad n\in \mathbb{N}.
\]
Now assume that $x=T(x)$ for some $x=\sum_{i=1}^\infty t_i x_i \in K$. Fix $n\in \mathbb{N}$. Then
\[
\| x - T_{\alpha_n}(x)\| \leq  \lim_{m\to \infty}\| T_{\alpha_m}(x) - T_{\alpha_n}(x)\|\leq \frac{1-\delta}{2}|1 - \alpha_n|.
\]
On the other hand,  
\[
x - T_{\alpha_n}(x) = \sum_{i=1}^\infty \big( t_{i+1} - (1-\delta)t^{\alpha_n}_i \big) x_{i+1},
\]
and hence
\[
\| x - T_{\alpha_n}(x)\| \geq (1-\delta) \sup_{i\in \mathbb{N}}\big| t_{i+1} - (1-\delta)t^{\alpha_n}_i \big|.
\]
Letting $n\to \infty$, we deduce that $x = \sum_{i=1}^\infty (1-\delta)^i x_i$.  Finally, since $\mathrm{d}(T, K)=0$, by (\ref{eqn:3sec6}) we have
\[
\mathrm{d}(T_\alpha, K) \leq (1-\delta)(1- \alpha) \sup_{0< t<1} t^\alpha|\ln t|.
\]
This proves part (1) of the theorem. 

\medskip 

(2) Let $\delta\in (0,1)$ be fixed. Pick $r>0$ so that $(2r)^{1-\alpha}\leq \lambda(1- \delta)$. We now apply Lemma \ref{thm:J}, part (ii), to get a sequence $(x_n)_{n=1}^\infty$ in $X$ so that
\[
(1-\delta)\sup_{n\geq 1}|a_n|\leq \Bigg\| \sum_{n=1}^\infty a_n x_n \Bigg\| \leq \sup_{n\geq 1}
|a_n|,
\]
for all sequence $(a_n)_{n=1}^\infty \in \co$. Let $(\alpha_n)_n$ be as in (1) and put $\beta_n:= 1- \alpha_n$. Set
\[
K=\Bigg\{ \sum_{n=1}^\infty t_n x_n \colon \text{each } 0 \leq t_n \leq r\Bigg\}. 
\]
Next we define an affine mapping $F\colon K\to K$ by
\[
F\Bigg( \sum_{n=1}^\infty t_n x_n \Bigg) = \sum_{n=1}^\infty \alpha_n  t_n x_n + r\sum_{n=1}^\infty \beta_n x_n. 
\]
Clearly $K\in \mathcal{B}(X)$ and $F$ is fixed point free. Further, it can be readily seen that
\begin{equation}\label{eqn:unif}
\| F^n(x) - F^n(y)\|  \leq \frac{(2r)^{1-\alpha}}{1 - \delta}\| x - y\|^\alpha\leq \lambda \| x - y\|^\alpha,
\end{equation}
for all $x, y\in K$ and $n\in \mathbb{N}$. This finishes the proof of (2). 

\medskip 

Let us prove (3). Take $\delta$ and $(x_n)_{n=1}^\infty$ to be as in (2). Since $X$ is separably Sobczyk, there is a projection $P\colon X\to [x_i]$. Denote by $[x_i]^+$ the positive cone of $[x_i]$ and consider the mapping $R_1\colon [x_i]\to [x_i]^+$ given by
\[
R_1\Bigg(\sum_{i=1}^\infty t_i x_i\Bigg)= \sum_{i=1}^\infty |t_i| x_i.
\]
Notice that $R_1$ defines a $(1-\delta)^{-1}$-Lipschitz retraction. Let $F$ and $K$ be as in (2), where $(1-\delta)^{-3}\|P\|(2r)^{1-\alpha}\leq \lambda$. Let us embed $[x_i]^+$ into $K$ through a $(1-\delta)^{-1}$-Lipschitz retraction map $R_2\colon [x_i]^+\to K$ given by
\[
R_2\Bigg( \sum_{i=1}^\infty t_i x\Bigg) = \sum_{i=1}^\infty \min(t_i, r) x_i.
\]
The fact that $R_2$ is $(1-\delta)^{-1}$-Lipschitz follows from the well-known fact (cf. \cite[p.1662]{DLT3}) ensuring that
\[
| s - t| = \big| \max\{s, r\} - \max\{ t, r\} \big| + \big| \min\{ s, r\} - \min\{ t,r\}\big|\,\,\,\text{for all } r,s, t\in \mathbb{R}.
\]
Finally, define $T:= F \circ R_2\circ R_1 \circ P$. It is easily shown from (\ref{eqn:unif}) that $T$ is uniformly $\alpha$-H\"older $\lambda$-contractive self-mapping of $B_X$ with $\digamma(T)=\emptyset$ and $\mathrm{d}(T,B_X)=0$. 
\end{proof}


\medskip 


Another open question in metric fixed point theory asks whether a Banach space containing a copy of $\ell_1$ fails the FPP for the class of asymptotically nonexpansive mappings. Our next result, which has a quantitative feature, solves this problem for the class of uniformly $\alpha$-H\"older nonexpansive maps. 

\begin{proposition}\label{R:2sec5} Let $X$ be a Banach space. Assume that $X$ contains an isomorphic copy of $\ell_1$. Then for every $\alpha\in (0,1)$ there exist a set $K\in \mathcal{B}(B_X)$ and a uniformly $\alpha$-H\"older nonexpansive affine mapping $T\colon K\to K$ such that $\digamma(T)= \emptyset$.
\end{proposition}

\begin{proof} (i) Fix a decreasing null sequence $(\delta_n)_{n=1}^\infty$ in $(0,1)$. Let $(x_n)_{n=1}^\infty$ be a sequence in $X$ satisfying inequality (\ref{eqn:A}). Let us set
\[
K=\Bigg\{ \sum_{i=1}^\infty t_i x_i \colon \text{each }\, t_i \geq 0\,\text{ and }\, \sum_{i=1}^\infty t_i =  2^{-1}\lambda^{\frac{1}{1-\alpha}}(1 - \delta_1)^{\frac{\alpha}{1 - \alpha}}\Bigg\}. 
\]
Then $K\in \mathcal{B}(B_X)$. Now take $T\colon K\to K$ to be the right shift on $(x_n)_{n=1}^\infty$, that is, $T(\sum_{i=1}^\infty t_i x_i)=\sum_{i=1}^\infty t_i x_{i+1}$. Since $T$ is affine, $\mathrm{d}(T, K)=0$. Also, note that $\digamma(T)=\emptyset$. Moreover, using (\ref{eqn:A}) we get that $T$ is uniformly $\alpha$-H\"older $\lambda$-contractive. 
\end{proof}

\medskip 

Proposition \ref{R:1sec4}, part (7), ensures for any $\alpha\in (0,1)$ that $B_{\ell_1}$  fails the FPP for uniformly $\alpha$-H\"older nonexpansive maps with null minimal displacement. This naturally raises the question of whether this fact remains true under equivalent renormings. The following  illustrates such a possibility.

\begin{example}\label{R:3sec5} $\ell_1$ admits an equivalent norm $\|\cdot\|$ with respect to which its unit ball fails the $\alpha$-H\"older FPP. Indeed, consider the norm $\| x\| =\max\{ \| x^+\|_1, \| x^-\|_1\}$ and set $K=\big\{ x=(t_i)_{i=1}^\infty\colon t_i\geq 0,\,\sum_{i=1}^\infty t_i \leq 1\big\}$. It is known that $K$ is weak$^*$-compact and convex, and $T\colon K\to K$ given by the equation $T(x) =\big( 1- \sum_{i=1}^\infty t_i, t_1, t_2, \dots \big)$ is a $\|\cdot\|$-isometry with $\digamma(T)=\emptyset$ (\cite[Example 1]{Lim}). In particular, by Proposition \ref{R:1sec4}, part (5), $K$ fails the $\alpha$-H\"older FPP for $\alpha\in (0,1)$. Of course, this isometry is also uniformly $\alpha$-H\"older nonexpansive. Finally, we observe that $R\colon B_{(\ell_1, \|\cdot\|)} \to K$ given by $R\big((t_i)_{i=1}^\infty \big)= ( t^+_i)_{i=1}^\infty$ defines a nonexpansive retraction. Thus $T\circ R \colon B_{(\ell_1, \|\cdot\|)} \to K$ is uniformly $\alpha$-H\"older nonexpansive and has no fixed points. 
\end{example}

\medskip 

To finish up this section we highlight the following general result which follows by using the arguments of Proposition \ref{R:1sec4}-(8) in concert with those of Proposition \ref{R:2sec5}, so we omit its proof.

\begin{proposition}\label{R:4sec5} Assume that $X$ contains a complemented copy of $\ell_1$. Then for any $\alpha\in (0,1)$ there is a uniformly $\alpha$-H\"older nonexpansive mapping $T\colon B_X\to B_X$ with $\digamma(T)=\emptyset$ and $\mathrm{d}(T, B_X)=0$. 
\end{proposition}
 
 \smallskip 

\begin{corollary} Let $X$ be a Banach space with an unconditional basis. Assume that $X$ is not reflexive. Then for any $\lambda, \alpha\in (0,1)$, $B_X$ fails the FPP for uniformly $\alpha$-H\"older $\lambda$-contractive mapping with null minimal displacement. 
\end{corollary}

\begin{proof} Assume that the basis of $X$ fails to be shrinking. Then (cf. \cite[Theorem 3.3.1]{AK}) $X$ contains a complemented subspace isomorphic to $\ell_1$. In this case the result follows directly from Proposition \ref{R:4sec5}. Assume, on the other hand, that the basis of $X$ is shrinking. Since $X$ is not reflexive, by a classical result due to R. C. James, it is not boundedly complete. By \cite[Theorem 1.c.10 (iii) $\Rightarrow$ (i)]{LT}, there is a seminormalized block basic sequence $(x_n)_{n=1}^\infty$ such that $\sup_m\|\sum_{i=1}^m x_i \|<\infty$. This and the unconditionality of the basis implies that $(x_n)_{n=1}^\infty$ is equivalent to the unit basis of $\co$. By Theorem \ref{R:1sec5}--(3) the result follows. 
\end{proof}

\smallskip 
\begin{corollary} Let $p\in [1, \infty]$. For any $\alpha\in (0,1)$ the following statements hold:
\begin{itemize}
\item[(i)] $B_{L_p}$ fails the FPP for uniformly $\alpha$-H\"older nonexpansive mappings with null minimal displacements provided $p\in \{ 1, \infty\}$. 
\item[(ii)] $B_{L_p}$  fails the FPP for $\alpha$-H\"older nonexpansive mappings with null minimal displacements whenever $1< p< \infty$. 
\end{itemize}
\end{corollary}

\begin{proof} (i) $\ell_1$ is isometrically isomorphic to a $1$-complemented subspace of $L_1$ (cf. \cite[p.110]{AK}), so Proposition \ref{R:4sec5} solves the case $p=1$. $L_\infty$ contains an isometric copy of $\ell_\infty$ (see \cite[p.85]{AK}), so $B_{\ell_\infty}$ is a hyperconvex metric subspace of $B_{L_\infty}$. By Theorem \ref{thm:AP} and Proposition \ref{R:7sec4} the result follows. (ii) By \cite[6.8, p.163]{AK} $\ell_2$ is linearly isometric to a complemented subspace $X$ of $L_p$. By Proposition \ref{R:8sec4} the result follows. 
\end{proof}


\medskip
\subsection{Open questions} We end this paper with the following questions: Let $X$ be an infinite-dimensional Banach space. Does $B_X$ fail the FPP for {\it uniformly} $\alpha$-H\"older nonexpansive maps with null minimal displacement if, (1) $X$ is reflexive, (2) $X$ is isomorphic to a Hilbert space or (3) $X$ is reflexive and has an unconditional basis, or then (4) $X$ is the James's space $J_2$? It should be noted that, intuitively, the general strategy used here to handle ($\mathcal{Q}3$) was to check the failure of the HFPP in a certain subset $K$ of $B_X$ and then try to properly extend it to the whole ball. In some cases this involves the availability of extension results with prescribed conditions, which seems to be the main difficulty in dealing with such a problem.

\end{large}

\renewcommand{\refname}{

\textsc{Department of Mathematics, Federal University of Cear\'a, Av. Humberto Monte, Fortaleza, CE 60455-369, Brazil} \par
  \textit{E-mail address}: \texttt{cleonbar@mat.ufc.br} \par
  \medskip

\end{document}